\documentclass[letterpaper, 10 pt, conference]{ieeeconf}
\IEEEoverridecommandlockouts
\usepackage{mathrsfs}
\usepackage{textcomp}
\usepackage{cite}
\usepackage{amsmath,amssymb,amsfonts}
\usepackage{graphicx}
\usepackage{float}
\usepackage{textcomp}
\usepackage{xcolor}
\usepackage{hyperref}
\usepackage{algpseudocode}
\usepackage[ruled,vlined,linesnumbered]{algorithm2e}

\makeatletter
\newlength{\algorithmtopruleskip}
\let\lcss@old@algocf@pre@ruled\@algocf@pre@ruled
\def\@algocf@pre@ruled{\vskip\algorithmtopruleskip\lcss@old@algocf@pre@ruled}
\let\lcss@old@algocf@pre@algoruled\@algocf@pre@algoruled
\def\@algocf@pre@algoruled{\vskip\algorithmtopruleskip\lcss@old@algocf@pre@algoruled}
\makeatother

\usepackage[normalem]{ulem} 
\newtheorem{remark}{Remark}
\newtheorem{lemma}{Lemma}
\newtheorem{theorem}{Theorem}
\newtheorem{proposition}{Proposition}
\newtheorem{definition}{Definition}
\newtheorem{problem}{Problem}
\newtheorem{corollary}{Corollary}
\newtheorem{assumption}{Assumption}

\usepackage{bm}
\usepackage{comment}
\usepackage[normalem]{ulem}
\usepackage{cancel} 

\newcommand{\D}{\nabla_x} 
\begin{document}

\title{\LARGE \bf Time-Varying Perturbations of Contractive Systems With an Application to Safe Stabilization}

\author{Andreas Oliveira, \IEEEmembership{Graduate Student Member, IEEE},
Arya Honarpisheh, \IEEEmembership{Graduate Student Member, IEEE},\\
Mustafa Bozdag, \IEEEmembership{Graduate Student Member, IEEE} and
Mario Sznaier, \IEEEmembership{Fellow, IEEE}
\thanks{This work was partially supported by NSF grant CMMI 2208182, AFOSR grant FA9550-19-1-0005 and  ONR grant N00014-21-1-2431. The authors are with the Robust Systems Lab, ECE Department, Northeastern University, Boston, MA 02115. \textit{(Corresponding author: Andreas Oliveira, e-mail: franciscodemelooli.a@northeastern.edu).}}
}

\maketitle
\thispagestyle{empty}
\pagestyle{empty}

\begin{abstract}
Perturbation theory for asymptotically stable systems has received sustained attention, leading to many criteria that preserve stability under uncertainty. A limitation of classical stability analysis, however, is that it is inherently equilibrium-dependent. Contractive systems, by contrast, are defined independently of any equilibrium and admit a rich set of robustness bounds. In this letter, we develop new time-varying perturbation conditions for contractive dynamics that preserve incremental exponential stability of the perturbed system and, in specific regimes, guarantee asymptotic convergence to solutions of the nominal dynamics. As an application, we propose a safety filter for steering a nominally contractive system to a desired equilibrium while avoiding an unsafe set, and we present an example illustrating advantages of the proposed method over existing approaches.
\end{abstract}

\begin{keywords}
Contractivity, perturbed systems, safe control, time-varying systems.
\end{keywords}

\section{Introduction}

Robustness of control systems is concerned with preserving stability and performance under uncertainties such as modeling errors and disturbances. For linear time-invariant systems, asymptotic stability is equivalent to the existence of a quadratic Lyapunov function, which yields robustness margins for perturbations that preserve the Lyapunov inequality; disturbances that do not preserve this inequality must typically vanish asymptotically to retain convergence to the equilibrium \cite{khalil2002nonlinear}. For nonlinear systems, the analysis is more delicate and is usually based on a Lyapunov function, whose existence may be guaranteed by converse Lyapunov theorems \cite{khalil2002nonlinear}. Robustness is then characterized by perturbations that preserve negativity of the Lyapunov derivative, or by vanishing perturbations, leading to notions such as practical stability \cite{teel1995tools} and input-to-state stability \cite{sontagISS}. Moreover, unlike in the linear case, even vanishing perturbations can produce unbounded solutions \cite{sontag2002remarkconverginginputconvergingstateproperty}.

To move beyond the equilibrium-dependent viewpoint attached to stability, robustness results have also been developed under contractivity assumptions. Originally introduced in the analysis of numerical integration methods \cite{dahlquist1958stability}, contraction theory has since found widespread applications in networked systems, biology, machine learning, data-driven control among others \cite{Sontag-Aminzare, FB-CTDS, oliveira2025DDC}. Informally, a system is contractive if a suitably defined distance between any two of its trajectories decays exponentially fast. Contractive systems possess several attractive structural properties. For example, for time-invariant systems, contraction implies the existence of a unique globally exponentially stable equilibrium, while for time-varying systems the notion remains trajectory-based and does not rely on equilibria \cite{FB-CTDS}. Moreover, contraction guarantees entrainment under periodic forcing \cite{contractive-with-inputs-sontag}. 

This equilibrium-independence is crucial for the analysis of perturbed contractive systems, as incremental exponential stability (IES) of the perturbed dynamics can be characterized without reference to the equilibrium of the nominal system. In this letter, contractivity is characterized through logarithmic norms, whose sub-additivity yields simple sufficient conditions for robustness of contraction under perturbations \cite{FB-CTDS}. Earlier bounds in \cite{coppel1965stability} showed asymptotic convergence of perturbed trajectories to nominal ones. More recently, \cite{contractive-with-inputs-sontag} established exponential convergence when both the perturbation and its differential vanish exponentially fast.

In this letter, we derive new perturbation bounds for contractive systems. We extend existing perturbation analyses by identifying a class of disturbances under which IES is preserved, without requiring the perturbed system to converge to the trajectories of the unperturbed system. For this class, we characterize when the perturbed system remains IES while converging to solutions of a different nominally contractive system, leveraging the Alekseev-Gröbner formula. In particular, if the disturbance converges uniformly to a limit such that the resulting vector field is exponentially stable, and its differential satisfies our disturbance condition, then the perturbed system remains IES and all trajectories converge to the new equilibrium. Likewise, if the nominal system is periodic and the disturbance converges to a periodic function whose differential satisfies the same condition, then the perturbed system remains IES and all trajectories converge to a periodic solution. 

As an application, we illustrate how the perturbation framework above can be used for safe control of a nominally contractive system to a desired equilibrium point. The construction follows an equilibrium-tracking viewpoint: we assume access to a \emph{safe reference evolution} generated by a feedforward shift that steers the equilibrium along a path that avoids an unsafe set, in the spirit of equilibrium tracking solutions studied in \cite{davydov2025time}. The resulting algorithm departs from traditional Control Lyapunov Function (CLF) -- Control Barrier Function (CBF) - Quadratic Programs (QP) \cite{Ames-Survey} in two ways. First, it uses a different objective: among all inputs that satisfy the barrier constraint, we select the one that is maximally aligned with the ``shifted'' nominal flow, thereby keeping the closed-loop direction close to the safe reference. 
Second, convergence to the desired point is not imposed through an explicit CLF inequality; it follows from the perturbation properties of the nominally contractive dynamics under a vanishing schedule on the feedforward shift and the feedback correction. The online computation is therefore devoted solely to enforcing safety.

\section{Preliminaries}

We denote the $n$-dimensional Euclidean space by $\mathbb{R}^n$. For a vector $x \in \mathbb{R}^n$ and some norm $| \cdot |$ in $\mathbb{R}^n$, $|x|$ denotes its given norm and $|x|_p$ if $|\cdot|$ is an $\ell_p$ norm. For a matrix $A$, $\|A\|$ denotes the induced matrix norm. The transpose of a matrix $X$ is denoted by $X^\top$. The Jacobian of $f$ with respect to $x$ is denoted by $\nabla_x f$, the derivative of a scalar function $\gamma$ by $\gamma'(t)$, and the inner product of two vectors $x$ and $y$ by $\langle x,y\rangle$.

\begin{definition}\cite{contractive-with-inputs-sontag}
Let $K>0$ and fix a norm $|\cdot|$ on $\mathbb{R}^n$. A subset $C\subset \mathbb{R}^n$ is \emph{$K$-reachable} if, for any two points $x_0,y_0\in C$, there exists a continuously differentiable curve $\gamma:[0,1]\to C$ such that
\begin{equation}
\gamma(0)=x_0,\, \gamma(1)=y_0,\, |\gamma'(r)| \le K\,|y_0-x_0| \, \forall r\in[0,1].    
\end{equation}
\end{definition}

\begin{definition}
A scalar continuous function $\alpha(r)$, defined for $r \in [0,a)$, is said to belong to class $\mathcal{K}$ if it is strictly increasing and $\alpha(0)=0$. It is said to belong to class $\mathcal{K}_{\infty}$ if it is also defined for all $r \ge 0$ and $\alpha(r)\to\infty$ as $r\to\infty$.
\end{definition}

We consider the (possibly) time-varying nonlinear system:
\begin{equation}\label{eq:sys}
    \dot{x}(t) = f(x,t), \qquad x \in \mathbb{R}^n, \ t \geq t_0,
\end{equation}
where $f:\mathbb{R}^n \times [0,\infty) \to \mathbb{R}^n$ is assumed to be continuously differentiable w.r.t. $x$ (unless otherwise specified) and continuous w.r.t. $t$. For any initial time $t_0\ge 0$ and initial condition $x_0$, we denote the corresponding solution by $\varphi(t;t_0,x_0)$.

\begin{definition}\label{def:IES}
System \eqref{eq:sys} is said to be \emph{incrementally exponentially stable} (IES) with overshoot $K$ on a forward-invariant set
$\Omega \subseteq \mathbb{R}^n$ if there exist $c>0$ such that, for every
$t_0\ge 0$ and any two solutions $\varphi(t;t_0,x_0)$ and $\varphi(t;t_0,y_0)$ of \eqref{eq:sys} with
$x_0,y_0\in\Omega$,
\begin{equation}
\label{eq:ies}
    |\varphi(t;t_0,x_0)-\varphi(t;t_0,y_0)|
    \le K\,e^{-c(t-t_0)}\,|x_0-y_0|,
    \forall\, t\ge t_0.
\end{equation}
\end{definition}

\subsection{Contraction Theory}

In this section, we introduce the class of contractive systems w.r.t. logarithmic norms and highlight several interesting properties that contractive systems present.

\begin{definition}{\cite{desoer2009feedback}} Let $|\cdot|$ be a norm on $\mathbb{R}^n$. The \emph{logarithmic norm} of a matrix $A \in \mathbb{R}^{n \times n}$ is defined by
\begin{equation}
    \mu(A) := \lim_{h \to 0^+} \frac{\|I + hA\| - 1}{h}.
\end{equation}
\end{definition}

\begin{definition}{\cite{contractive-with-inputs-sontag}}
The system \eqref{eq:sys} or its vector field, is said to be \emph{infinitesimally contracting} in a set $G \subset \mathbb{R}^n$, if there exists some norm $|\cdot|$, with associated logarithmic norm $\mu$,  such that, for some constant $c > 0$ (the contraction rate), it holds that:
\begin{equation}\label{def:infinitesimally-contracting}
    \mu(\D f(x,t)) \leq -c,  \quad \forall x \in G \quad \forall t \ge 0 .
\end{equation}
\end{definition}

This definition based on the logarithmic norm of the Jacobian of $f$ yields the convergence behavior:

\begin{theorem}{\cite{contractive-with-inputs-sontag}}\label{thm1}
Suppose that $f(x,t)$ is infinitesimally contracting, with contraction rate $c$, on a $K$-reachable forward invariant set $G$. Then, for every two solutions $\varphi(t;t_0,x_0)$ and $\varphi(t;t_0,y_0)$ of \eqref{eq:sys} such that $x_0,y_0 \in G$, it holds that:
\begin{equation}
    |\varphi(t;t_0,x_0) - \varphi(t;t_0,y_0)| \leq Ke^{-c(t-t_0)} \, |x_0 - y_0|, 
    \forall t \geq t_0.
\end{equation}
\end{theorem}

Therefore ensuring $\mu(\D f(x,t))$ is uniformly bounded above by a negative constant is a sufficient condition for IES of a forward invariant $K-$reachable set $G$. Other interesting properties follow from the above such as entrainment to periodic forcing, existence of exponentially stable equilibrium for time-invariant systems and others \cite{FB-CTDS}.

We use the following definition of Control Barrier Functions (CBFs) in combination with the contractivity results of this paper to develop an algorithm, which jointly guarantees safety and asymptotic convergence to a desired point.
\begin{definition}{\cite{Ames-Survey}}\label{def:CBF}
A continuously differentiable function $h:\mathbb{R}^n\to\mathbb{R}$ is 
a CBF for $\dot{x}=f(x,t)+g(x)u$ with respect to $\mathcal{S}=\{x:h(x)\ge 0\}$ if there exists $\alpha\in\mathcal{K}$ such that
\begin{equation}
    \sup_{u}\bigl[L_f h(x) + L_g h(x)\,u\bigr] \ge -\alpha(h(x)),\, \forall\,x\in\mathcal{S}, \, \forall t \ge t_0.
\end{equation}
\end{definition}

\section{Robustness of Contractive Systems}
First, we consider the perturbed version of \eqref{eq:sys}:
\begin{equation}\label{eq:perturbed-ode}
\dot{x}(t) = f(x,t) + g(x,t),
\end{equation}
where $g$ is also assumed to be continuously differentiable w.r.t. $x$ and continuous w.r.t. $t$. In classical contractivity perturbation analyses, the goal is to establish sufficient conditions on the disturbance term $g(x,t)$ under which the perturbed system \eqref{eq:perturbed-ode} preserves the infinitesimally contracting property. For this there are simple sufficient conditions, that leverage the subadditivity of the logarithmic norm, and can be found in Lemma 3 of \cite{contractive-with-inputs-sontag} and Theorem 3.19 of \cite{FB-CTDS}.

In this letter, we relax the requirement that the perturbed system remain infinitesimally contracting in order to accommodate a broader class of disturbances. Nevertheless, we retain IES of the resulting system. The main result in the literature satisfying these assumptions is given in \cite{contractive-with-inputs-sontag}. In particular, there it is shown that if the nominal dynamics are infinitesimally contracting on a $K$-reachable forward-invariant set and if a time-varying perturbation $g(x,t)$, together with its Jacobian $\D g(x,t)$, decays uniformly exponentially in time, then the perturbed trajectories converge exponentially to nominal trajectories, and the perturbed dynamics is approximately incrementally exponentially stable. 

The result above imposes relatively strong assumptions, namely that both the perturbation and its Jacobian decay exponentially in time. In this work, we relax these requirements to show a broader class of perturbations preserves IES of the perturbed system. We further show that exponential convergence to nominal trajectories can be recovered without requiring the exponential decay of the Jacobian.

\subsection{Time-Varying Perturbations}

In order to develop our perturbation results we split the perturbations into two different functions, namely we consider the following vector field:
\begin{equation}\label{eq:nonlinear-perturbation}
    \dot{x}(t) = f(x,t) + g(x,t) + h(x,t).
\end{equation}
where both $g,h$ are assumed to be continuously differentiable w.r.t. $x$ and continuous w.r.t. $t$. As we will show, $g(x,t)$ and $h(x,t)$ can exhibit qualitatively distinct behaviors while still preserving IES of the resulting system. Before stating our results, it is important to point out that IES for a time-varying system does not, in general, imply the existence of an exponentially stable equilibrium, unlike in the time-invariant case. Consequently, unless one enforces stringent structural conditions on $g(x,t)$ and $h(x,t)$, it is impossible to characterize the solutions beyond the defining property that any pair of trajectories converges exponentially to one another. We first establish a technical lemma used in the proof of the main result:
\begin{lemma}\label{lemma:concentration}
    Assume $c > 0$, $\zeta(t)$ satisfies:
    \begin{equation}
        \sup_{t \in [0,\infty)} \int_0^{t} \zeta(s) \, ds \le \eta < \infty,
    \end{equation}
    and $\xi(t)$ is locally bounded and $\limsup_{t \to \infty} \xi(t) \to l \le c -\delta$ where $\delta \in (0,c)$. Then for any $\epsilon \in (0,\delta)$, there exists $\kappa$ such that:
    \begin{equation}\label{eq:uniformly-negative}
        \int_0^{t} -c + \zeta(s) + \xi(s) \; ds \le -(\delta -\epsilon)t + \kappa, \qquad \forall t \ge 0.
    \end{equation}
\end{lemma}

\begin{proof}
    Let $\epsilon \in (0,\delta)$ since $\limsup_{t \to \infty} \xi(t) \to l$, there exists $N > 0$ such that $\xi(t) - l \le \frac{\epsilon}{2}$ for all $t \ge N$. Let $T:=\max\left\{\frac{2\eta}{\epsilon},N\right\}.$ Then, for all $t\ge T$,
    \begin{equation}
    \begin{aligned}
        & \int_0^{t}\bigl(-c+\zeta(s) + \xi(s)\bigr)\,ds \\
        &\le -c t + \eta + \int_0^{N} \xi(s)\,ds
        + \int_{N}^{t} \xi(s)\,ds \\
        &\le -c t + \frac{\epsilon}{2} t
        + \int_0^{N} \xi(s)\,ds
        + \left(l + \frac{\epsilon}{2}\right)(t-N) \\
        &\le
        -(\delta - \epsilon)t
        + \int_0^N \xi(s)\,ds
        - \left(c-\delta + \frac{\epsilon}{2}\right)N .
    \end{aligned}
    \end{equation}
    Let
    \begin{equation}
        \kappa_1 :=
        \int_0^{N} \xi(s)\,ds
        - \left(c -\delta + \frac{\epsilon}{2}\right)N,
    \end{equation}
    which is finite since $\xi(t)$ is locally bounded. 
    Moreover, define
    \begin{equation}
        \kappa_0 :=
        \max_{t \in [0,T]}
        \left\{
        \int_0^{t} \bigl(-c + \zeta(s) + \xi(s)\bigr)\,ds
        + (\delta - \epsilon)t
        \right\}.
    \end{equation}
    Finally, set $\kappa=\max\{\kappa_0,\kappa_1\}$. Then
    \begin{equation}
        \int_0^{t} \bigl(-c + \zeta(s) + \xi(s)\bigr)\,ds
        \le -(\delta - \epsilon)t + \kappa,
        \qquad \forall t \ge 0.
    \end{equation}
\end{proof}
With Lemma~\ref{lemma:concentration}, we are ready to state our main Proposition:

\begin{proposition}\label{Prop:Nonlinear}
Consider \eqref{eq:nonlinear-perturbation} and assume that the nominal system \eqref{eq:sys} is infinitesimally contracting with contraction rate $c>0$ in a $K$-reachable set $G$ with respect to a logarithmic norm $\mu$. If the perturbed system is forward-invariant in $G$ and the perturbations satisfy
\begin{equation}\label{eq:concentration-on-perturbation}
    \zeta(t) = \sup_{x \in G} \mu(\D g(x,t)), \quad \sup_{t \in [t_0,\infty)}\int_{t_0}^{t} \zeta(s) ds < \infty,
\end{equation}
and if there exists $\delta$ satisfying $0<\delta<c$ such that
\begin{equation}\label{eq:sub-expansivity}
    \xi(t) = \sup_{x \in G} \mu (\D h(x,t)), \quad  \limsup \xi(t) \to l \le c-\delta
\end{equation}
and $\xi(t)$ is locally bounded. Then for any $x_0,z_0 \in G$ and the associated pair of solutions 
$\varphi(t;t_0,x_0)$ and $\varphi(t;t_0,z_0)$ of \eqref{eq:nonlinear-perturbation}, 
there exist constants $K_p,\lambda_p>0$ such that
\begin{equation}
    |\varphi(t;t_0,x_0)-\varphi(t;t_0,z_0)|
    \le K_p \,e^{-\lambda_p (t-t_0)}\,|x_0-z_0|
\end{equation}
for all $t\ge t_0$, and hence \eqref{eq:nonlinear-perturbation} is \emph{IES} on $G$.
\end{proposition}

\begin{proof}
    The proof below follows the outline of the proof in \cite{contractive-with-inputs-sontag} along with some steps to show that despite the perturbations IES of the system is preserved.  We know that since $f(x,t)$ is infinitesimally contracting on $G$ with contraction rate $c>0$, $\mu(\D f(x,t))\le -c$ for all $x \in G, t \ge t_0$. Consider two solutions $\varphi(t;t_0,x_0)$ and $\varphi(t;t_0,z_0)$ of $\dot{x}=f(x,t) + g(x,t) + h(x,t)$. Choose a $C^1$ curve $\gamma:[0,1]\to G$ with $\gamma(0)=z_0$, $\gamma(1)=x_0$, and $|\gamma'(r)|\le K\,|x_0-z_0|$ for all $r\in[0,1]$. Define $\psi(t,r):=\varphi(t;t_0,\gamma(r))$ and $w(t,r):=\frac{\partial \psi}{\partial r}(t,r)$.
    Differentiating and using the chain rule,
    \begin{equation}
    \small
    \begin{aligned}
    & \frac{\partial w}{\partial t}(t,r)
    = \frac{\partial}{\partial r}\Big(
        f(\psi(t,r),t)
    + g(\psi(t,r),t)
        + h(\psi(t,r),t)\Big) \\
    &= \Big(
        \D f(\psi(t,r),t)
        + \D g(\psi(t,r),t)
     + \D h(\psi(t,r),t)
        \Big) w(t,r).
    \end{aligned}
    \end{equation}
    By Coppel's inequality \cite{coppel1965stability} and sub-additivity of logarithmic norms \cite{desoer2009feedback} (Chapter 2):
    \begin{equation}
    \begin{aligned}
    & |w(t,r)| \le |w(t_0,r)| \exp \; \!\Bigl(\int_{t_0}^{t} \mu(\D f(\psi(\tau,r), \tau))+\\
     &\mu(\D g(\psi(\tau,r),\tau)) +\mu(\D h(\psi(\tau,r),\tau))\,d\tau\Bigr).
    \end{aligned}
    \end{equation}
    From assumptions \eqref{eq:concentration-on-perturbation},\eqref{eq:sub-expansivity} and Lemma \ref{lemma:concentration}, for any $\epsilon \in (0,\delta)$ there exists $\kappa$ such that
    \begin{equation}
    \begin{aligned}
            & \int_{t_0}^{t} \mu\!\big( \D f (\psi(\tau,r))) + \mu(\D g(\psi(\tau,r),\tau)) + \\ 
            &  \mu(\D h(\psi(\tau,r),\tau)) \, \mathrm d\tau < -(\delta - \epsilon)(t-t_0) + \kappa.
    \end{aligned}
    \end{equation}
    Hence by the Fundamental Theorem of Calculus,
    \begin{equation}
    \varphi(t;t_0,x_0)-\varphi(t;t_0,z_0)=\psi(t,1)-\psi(t,0)=\int_0^1 w(t,s)\,ds,
    \end{equation}
    so
    \begin{equation}
    \begin{aligned}
     & |\varphi(t;t_0,x_0)-\varphi(t;t_0,z_0)| \le 
     \int_0^1 |w(t,s)| \, ds\ \\
     & \le e^{-(\delta -\epsilon)(t-t_0) + \kappa} \int_0^1 |w(t_0,s)|ds \\
     & = e^{-(\delta -\epsilon)(t-t_0) + \kappa} \int_0^1 |\gamma'(s)|ds \\
     & \le e^\kappa K e^{-(\delta -\epsilon)(t-t_0)} |x_0 - z_0|
    \end{aligned}
    \end{equation}
    for all $t \ge t_0$.
\end{proof}

\begin{remark}
Before stating the next corollary, we highlight a key distinction from Lyapunov-based stability analysis: infinitesimal contraction is invariant to the addition of terms $g(x,t)$ satisfying $\mu(\D g(x,t)) \le c - \epsilon$, where $c$ is the nominal contraction rate and $\epsilon > 0$, in the same logarithmic norm \cite{FB-CTDS} (e.g., time-varying disturbances $g(x,t) = a(t)$). Thus, if $\dot{x} = f(x,t) + \delta(x,t)$, where $\delta(x,t)$ converges uniformly to a contraction-preserving vector field $g(x,t)$, then the perturbed dynamics can be written as $\dot x = f(x,t) + g(x,t) + \delta'(x,t)$, where $\delta'(x,t) \to 0$ and $f(x,t) + g(x,t)$ is infinitesimally contracting. Hence, for our purposes, the following Corollary assumes $\delta(x,t)$ converges uniformly to $0$.
\end{remark}

\begin{corollary}\label{cor1}
Assume $f(x,t)$ is infinitesimally contracting with contraction rate $c > 0$ on a forward invariant, $K$-reachable set $G \subset \mathbb{R}^n$. Consider a perturbation $\delta(x,t)$ that is continuously differentiable w.r.t. $x$ and continuous in $t$, which preserves forward invariance of the set $G$. Further, let $\gamma(t) := \sup_{x \in G} |\delta(x,t)|$, assume that $\gamma(t)$ is bounded and $ \lim_{t \to \infty} \gamma(t)= 0$ and $\delta(x,t)$ satisfies either \eqref{eq:concentration-on-perturbation} or \eqref{eq:sub-expansivity}. Then $\dot{x} = f(x,t) + \delta(x,t),$ is incrementally exponentially stable on $G$, with all trajectories converging asymptotically to any solution of $f(x,t)$. 
\end{corollary}

\begin{proof}
By assumption, $G$ is forward invariant for the dynamics induced by $f(x,t) + \delta(x,t)$ and given that $\delta(x,t)$ satisfies either \eqref{eq:concentration-on-perturbation} or \eqref{eq:sub-expansivity} the perturbed system is IES on $G$ by Proposition \ref{Prop:Nonlinear}. 
Let $x_c(t) = \varphi(t;t_0,x_0)$ be the solution of $\dot{x} = f(x,t)$ for some initial condition $x_0$ and $x_p(t)= \psi(t;t_0,x_0)$ be the solution of the perturbed system $\dot{x} = f(x,t) + \delta(x,t)$ with initial condition $x_0$. By the Alekseev-Gröbner (AG) formula (Lemma 3 in \cite{brauer1966perturbations}):
\begin{equation}\label{eq:alekseev-formula}
\begin{aligned}
    x_p(t) - x_c(t) = \int_{t_0}^{t} \D \varphi (t;\tau,x_p(\tau))\delta(x_p(\tau),\tau) d\tau.
\end{aligned}
\end{equation}
Since the nominal system is infinitesimally contracting on the $K-$reachable set $G$ with contraction rate $c$ by Theorem \ref{thm1}: 
\begin{equation}
    |\varphi(t;\tau,x) - \varphi(t;\tau,y)| \le Ke^{-c(t-\tau)}|x-y|.
\end{equation}
Hence the function $\D \varphi(t;\tau,x_p(\tau))$ satisfies in $G$:
\begin{equation}
    \|\D \varphi(t;\tau,x_p(\tau))\|  \le Ke^{-c(t-\tau)}.
\end{equation}
The perturbed solution $x_p(\tau)$ is assume to be forward-invariant in $G$ therefore:
\begin{equation}\label{eq:nonlinear-variation-of-constants}
\begin{aligned}
    & |x_p(t) - x_c(t)| \le 
    \int_{t_0}^{t} \Bigl(\|\D \varphi (t;\tau,x_p(\tau))\|\\
    &  |\delta(x_p(\tau),\tau)| d\tau \Bigr) \le Ke^{-ct}\int_{t_0}^t e^{c\tau} |\gamma(\tau)| d\tau.
\end{aligned}
\end{equation}

\noindent By assumption $\gamma(\tau) \to 0$, so for any $\epsilon > 0$ there exists $N$ such that for all $\tau > N$, $\gamma(\tau) < \frac{c\epsilon}{2}$. Since $\gamma$ is bounded defining $M= \sup_{\tau \le N} |\gamma(\tau)|$ we have:
\begin{equation}
\begin{aligned}
   & e^{-ct}\int_{t_0}^t e^{c\tau} |\gamma(\tau)| d\tau \le e^{-ct}\int_{t_0}^{N}e^{c\tau}Md\tau \\
   & + e^{-ct}\int_{N}^{t} \frac{c\epsilon}{2} e^{c\tau}d\tau \le \frac{Me^{-ct}}{c}(e^{cN} -e^{ct_0}) + \frac{\epsilon}{2}.
\end{aligned}
\end{equation}

\noindent Choosing $T$ such that 
\begin{equation}
Me^{-cT}\left(e^{cN} - e^{ct_0}\right)/c \le \epsilon/2
\end{equation} shows that for $t > \max\{N,T\}$ we have $|x_p(t) - x_c(t)| \le K\epsilon$ and since $\epsilon$ was arbitrary and $K$ is a fixed constant $x_p(t) \to x_c(t)$. Since the nominal system is contractive a triangle inequality shows $x_p(t)$ approaches any solution of the nominal system.
\end{proof}

\begin{remark}\label{rm:alekseev-grobner}
The special case in which the perturbation decays exponentially uniformly on $G$, that is, $\sup_{x\in G}|\delta(x,t)| \le L_2 e^{-c_2 t}$ for some $L_2,c_2>0$, is closely related to a simple extension of Theorem~3 in \cite{hamadeh2015contraction} to handle vanishing disturbances and Corollary 3.17 of \cite{FB-CTDS}. In that setting, and when $G$ is convex, one obtains exponential convergence of the perturbed trajectory to the corresponding nominal trajectory. Lemma~4 in \cite{contractive-with-inputs-sontag} generalizes such to when G is $K-$reachable by imposing further that the differential of the perturbation is also exponentially decaying, meaning there exists $L_3,c_3 > 0$ such that $\sup_{x\in G}\|\D \delta(x,t)\| \le L_3 e^{-c_3 t}.$ We note that the AG formula gives a slightly sharper version of the two results above: for exponential (and asymptotic) convergence of the perturbed trajectory to nominal trajectories, it is not necessary to assume convexity of $G$, an exponential (or asymptotic) bound on the differential of the perturbation and an infinitesimally contracting nominal system. Indeed, suppose that both the nominal and perturbed systems are forward invariant on the set $G$, that the nominal system is IES (\emph{this includes infinitesimally contracting vector fields but is not limited to it}) on $G$ with rate $c>0$ and overshoot $K$, that $\delta(x,t)$ is continuous in $x$ \emph{(differentiability is not needed)} and in $t$, further $\sup_{x\in G}|\delta(x,t)| \le L_2 e^{-c_2 t}$ (the proof of the asymptotic case follows from the Corollary). Then, for any nominal solution $\varphi(t;t_0,x_0)$ and any perturbed solution $\psi(t;t_0,x_1)$, the triangle inequality and the AG formula yield
\begin{equation}
\begin{aligned}
    & |\varphi(t;t_0,x_0)-\psi(t;t_0,x_1)| \le \\
    & |\varphi(t;t_0,x_0)-\varphi(t;t_0,x_1)|  + |\varphi(t;t_0,x_1)-\psi(t;t_0,x_1)|  \\
    &\le K e^{-c(t-t_0)}|x_0-x_1|
    + C_\rho e^{-\rho(t-t_0)},
\end{aligned}
\end{equation}
for every $\rho<\min\{c,c_2\}$ and for some constant $C_\rho>0$. Consequently,
\begin{equation}
    |\varphi(t;t_0,x_0)-\psi(t;t_0,x_1)|
    \le \bar C_\rho \bigl(1+|x_0-x_1|\bigr)e^{-\rho(t-t_0)}
\end{equation}
for some $\bar C_\rho>0$. This has a rich history in the theory of stability of variations which is studied in \cite{brauer1966perturbations}.
\end{remark}

\begin{remark}
     Lastly, we note that if $f(x,t) $ has a unique equilibrium point then $f(x,t) + \delta(x,t)$ is IES and converges asymptotically (and exponentially if $\delta(x,t) \to 0$ exponentially) to such equilibrium point and if $f(x,t)$ is periodic then the perturbed system is IES and converges asymptotically (and exponentially if $\delta(x,t) \to 0$ exponentially) to the periodic solution of $\dot{x}(t) = f(x,t)$.
\end{remark}

\section{Safety Via Vanishing Controllers}
We apply the tools above to design a time-varying feedback law $u(x,t)$ together with a feedforward input $l(t)$ for the control-affine system
\begin{equation}\label{eq:safe-closed-loop-system}
    \dot{x} = f(x) + g(x)u(x,t) + l(t),
\end{equation}
so that trajectories avoid an unsafe set $\mathcal U$. We implement the design as a safety filter partially satisfying the conditions of Corollary~\ref{cor1}.

\begin{assumption}
Let $f$ be continuously differentiable and infinitesimally contracting in $\mathbb{R}^n$, w.r.t. a norm $|\cdot|$, which is differentiable away from the origin.  Define the compact set $B_r(0):=\{x:|x|\le r\}$ and assume $f(0)=0$. Let the unsafe set be $\mathcal U := \{x\in\mathbb{R}^n : h_U(x) < 0\},$ where $h_U:\mathbb{R}^n\to\mathbb{R}$ is continuously differentiable, and the safe region $\mathcal S := B_r(0)\setminus \mathcal U$.
\end{assumption}

\begin{problem}\label{prob1}
Design $u(x,t)$ and $l(t)$ to drive all trajectories initialized in a ball $\mathcal{X}_0\subset \mathcal{S}$ to the origin and avoiding $\mathcal U$. 
\end{problem}

By exploiting global contractivity of the nominal vector field $f$, we can steer any initial condition in $B_r(0)$ toward a desired point $x_r$ by applying the constant feedforward input $l(t)\equiv -f(x_r)$. This renders $x_r$ an equilibrium of the shifted dynamics, although it does not necessarily preserve safety. If there exists a safe path connecting $x_r$ to the origin, i.e., a curve that does not intersect $\mathcal U$, then we may gradually drive the feedforward input to zero, $l(t)\to 0$, while keeping the resulting trajectory close to the reference evolution induced by $l(t)$, as in the equilibrium-tracking constructions studied in \cite{davydov2025time}. While choosing such an $x_r$ is straightforward in low-dimensional examples, it becomes nontrivial in high dimensions. A practical approach is to compute reference points $x_r$ that admit simple safe connectors to the origin (e.g., collision-free straight-line paths) and are ``close'' to the set of initial conditions. One way to build such safe connectors in high dimensional spaces is presented in \cite{kavraki2002probabilistic}. Even if such an $x_r$ is available, the main issue remains: how do we ensure that \emph{every} trajectory initialized in a ball $\mathcal{X}_0$ follows the reference path induced by $l(t)$ safely? To that end, we introduce a time-vanishing feedback $u(x,t)$ together with a safety filter that enforces the CBF constraint along the closed-loop evolution. A key distinction from standard minimum-norm CLF--CBF QPs is that, among  all feasible safe inputs, we select one that maximally aligns with the direction of the vector field $f(x)+l(t)$, as enforced by the minimization objective in \eqref{eq:qp-vanishing-cbf-only}. In this way, the closed loop system is encouraged to track the safe path coordinated by $l(t)$. Lastly an $\ell_2$ bound on the state-feedback component of the control law is enforced to guarantee that the former vanishes in the limit, making the resulting problem a Quadratically Constrained Quadratic Program (QCQP). This is summarized in Algorithm~\ref{alg:vanishing-clf-cbf-qp}. 

\vspace*{2pt}
\begin{algorithm}
\DontPrintSemicolon
\caption{Vanishing CBF-QCQP Safety Filter}
\label{alg:vanishing-clf-cbf-qp}
\footnotesize

\KwIn{Dynamics $\dot x=f(x)+g(x)u(x,t)+l(t)$ with $(x_0,t_0)$; \;

Unsafe set $\mathcal U=\{x:h_U(x)<0\}$; Ball $B_r(0)=\{x:|x|\le r\}$; \;
Safe path $l(t)$ steering the equilibrium from $x_r$ to $0$; \;
$\kappa_B,\kappa_U \in \mathcal{K}$; \quad $\alpha(t) \in \mathcal{K}_{\infty}$; \quad $u_{max}$.}
\KwOut{$u(x,t)$.}

\BlankLine
\textbf{Define smooth barriers for staying inside} $B_r(0)$:\;
$h_B(x)\gets r-|x|$ \tcp*{$B_r(0)=\{h_B\ge 0\}$}
\While{$t\ge t_0$}{
    Measure $x\gets x(t)$, set $s(t)\gets 1+\alpha(t)$,\;
    Compute $k^\star(x,t)$ by solving:
    \begin{equation}
    \label{eq:qp-vanishing-cbf-only}
    \begin{aligned}
    & \min_{k\in\mathbb{R}^m} \frac{1}{2}|k|^2 - \langle f(x) + l(t), g(x) k\rangle \\
    \text{s.t.}\quad
    & L_f h_B(x) + L_g h_B(x)\,\frac{k}{s(t)} + \nabla h_B(x)^\top l(t) \\
     & \ \ge\ -\kappa_B\!\big(h_B(x)\big),\\
    & L_f h_{U}(x) + L_g h_{U}(x)\,\frac{k}{s(t)} + \nabla h_{U}(x)^\top l(t) \\
    & \ \ge\ -\kappa_U\!\big(h_{U}(x)\big), \\
    & |k|_{2}^2\le u_{max}
    \end{aligned}
    \end{equation}
    
    Apply $u(x,t)\gets \dfrac{k^\star(x,t)}{s(t)}$ \tcp*{vanishing: $u(x,t)\to 0$}
}
\end{algorithm}

The first two inequalities in \eqref{eq:qp-vanishing-cbf-only} come from applying the CBF Definition~\ref{def:CBF} for the ball $B_r(0)$ and unsafe set $\mathcal{U}$ applied to the closed-loop dynamics $\dot x = f(x) + g(x)k/s(t) + l(t)$, with class-$\mathcal{K}$ functions $\kappa_B,\kappa_U$ respectively. In general feasibility of Algorithm~\ref{alg:vanishing-clf-cbf-qp} is not guaranteed: it requires a $k$ satisfying both the CBF constraints and $|k|_2^2 \le u_{\max}$. Hence feasibility must be verified for a given initial condition $x \in \mathcal{X}_0$ after running the algorithm. When Algorithm \eqref{alg:vanishing-clf-cbf-qp} is indeed feasible, the optimizer $k^\star(x,t)$ is unique, since the objective is strongly convex in $k$ and the feasible set is convex for each fixed $(x,t)$. However, to justify  substituting $k^\star$ into the ODE~\eqref{eq:safe-closed-loop-system}, one must additionally  establish that it is sufficiently regular as a function of $(x,t)$. In this paper we simply assume that the optimization problem in Algorithm~1 satisfies the constraint qualifications required to derive locally Lipschitz controllers which is presented in section 3 of \cite{mestres2025regularity}. Lastly if Algorithm~1 is indeed feasible along the entire closed-loop trajectory given a initial condition in $\mathcal{X}_0$ and starting time $t_0$, the closed-loop system is forward-invariant in the bounded set $\mathcal{S}$. Using \(|k^\star(x,t)| \leq \sqrt{u_{\max}}\) and \(s(t)\to\infty\) we conclude the induced perturbation \(g(x)k^\star(x,t)/s(t)\) vanishes to $0$ along the closed-loop path. Therefore, considering the pairwise difference between the controlled and nominal systems for an initial condition in $\mathcal{X}_0$, following the notation of Corollary~\ref{cor1}, the perturbation $\delta(x_p(t),t) = g(x_p(t))u(x_p(t),t) + l(t)$, where $x_p(t)$ is now a solution of the closed-loop system, vanishes whenever Algorithm~1 is feasible. By the convergence argument in Corollary~\ref{cor1} (which also works for locally Lipschitz vector fields by Remark~\ref{rm:alekseev-grobner}) the closed-loop system converges to the origin.

\subsection{Example}

Consider the two-dimensional nonlinear system:
\begin{equation}\label{eq:example_system}
    \dot{x} = \begin{bmatrix}
        -2x_1 + x_2 - \sin(x_1) \\
        -2x_2 - x_1 + \cos(x_2) - 1
    \end{bmatrix} + 
    \begin{bmatrix}
    1 \\
    1
    \end{bmatrix}u(x,t) + l(t).
\end{equation}
The nominal system is contractive w.r.t. the $\ell_2$ norm and we consider the ball $B_{40}(0) = \{x \;| \; |x|_{2} \le 40\}$. The unsafe set $\mathcal{U}$ is defined by the polynomial function:
\begin{equation}
    \begin{aligned}
        h_U(x_1,x_2) &= 4\left(\frac{x_1}{1.5}\right)^{4} 
        - 20\left(\frac{x_1}{1.5}\right)^{2}(x_1-15) 
        - 13\left(\frac{x_1}{1.5}\right)^{2} \\
        &+ 25(x_2-15)^{2} + 35(x_2-15) - 2
    \end{aligned}
    \label{eq:h_U}
\end{equation}
which describes a crescent-shaped region and is illustrated by the red shaded object in Figure \ref{fig:trajectories}. Initial conditions are sampled randomly from the ball $\mathcal{X}_0 := \{x:|x-x_0|\le r_0\}$ with $x_0=[0,18]^\top$, $r_0=2$.

Running a CLF--CBF QP as in \cite{Ames-Survey}, using different quadratic Lyapunov functions \(V(x)=x^\top P x\), becomes infeasible for many initial conditions in $\mathcal{X}_0$. This is geometrically intuitive: any quadratic Lyapunov function must temporarily increase to escape the well.  One may relax the stability constraint in CLF--CBF QPs by augmenting the CLF inequality to the objective, so that safety can be enforced strictly and stability is hopefully preserved. However it has been shown in \cite{reis2020control} that even if the QP is feasible for all $x$, the resulting controller can introduce new asymptotically stable equilibria.

Density functions proposed in \cite{rantzer-safe-control} can also be applied to solve the problem with a polynomial approximation of $f(x)$. Since $h_U(x) \ge 0$ is semialgebraic, the problem can be posed as a sum-of-squares (SOS) program. However, getting a tight approximation of the dynamics entails using high-order polynomials, leading to semidefinite programs with combinatorial complexity and with no guarantee of feasibility for low-order truncations. In our case, using a 10\textsuperscript{th} order Taylor polynomial to approximate the sinusoidal terms in \eqref{eq:example_system} led to an infeasible SOS program.

Now running Algorithm 1 with the following simple parameter setup:
\begin{equation}
    \begin{aligned}
        & \kappa_B(h) = 5h, \;\;  \kappa_U(h) = h, \;\; u_{max}=100, \;\; \alpha(t) = t, \\ 
        & x_r = [-25, 25]^{\top}, \; l(t) = -f(x_r(1-t/T_{\text{sim}})), \; t_0 = 0,
    \end{aligned}
\end{equation}
where $T_{\text{sim}}$ is the simulation horizon and $f$ is the nominal vector field, results in a feasible program for every initial condition in $\mathcal{X}_0$. The resulting trajectories are displayed in Figure \ref{fig:trajectories} and show asymptotic convergence of every sampled initial condition in $\mathcal{X}_0$ to the origin.

\section{Conclusion}

In this letter, we presented novel perturbation bounds that preserve the IES of a nominally contractive system and, under additional conditions, guarantee convergence of the perturbed system to a different nominally contractive system. As an application, we showed how to design time-vanishing feedforward and feedback controllers that ensure safety and convergence to a desired point. In future work, we aim to leverage advances in motion planning, such as \cite{teshome2025real}, to extend this safety framework to more realistic settings, including evader--pursuer dynamics and time-varying unsafe sets \cite{bozdag2025pursuit}.

\begin{figure}[t]
    \centering
    \includegraphics[width=0.6\linewidth]{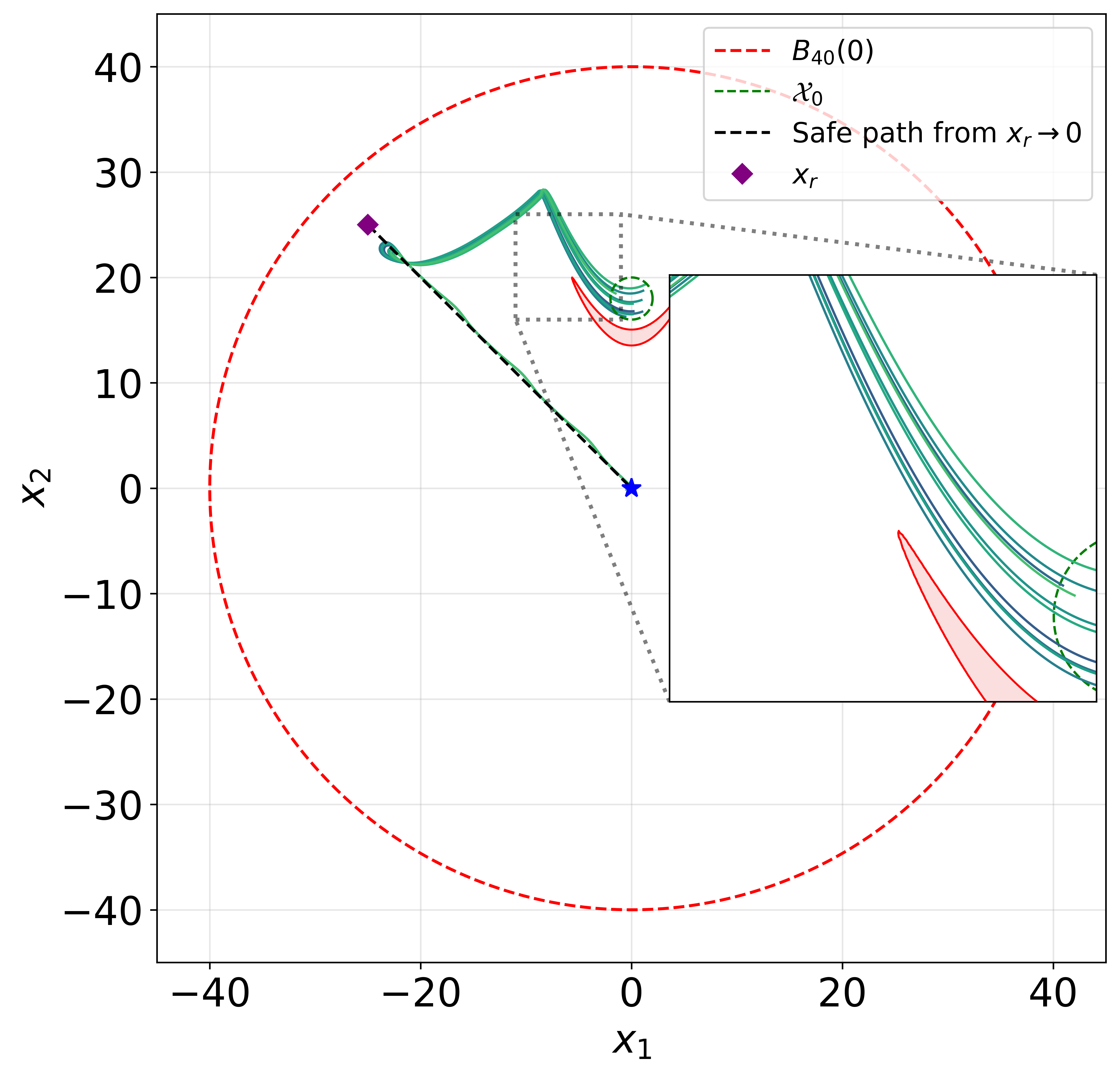}
    \caption{Algorithm~\ref{alg:vanishing-clf-cbf-qp} applied to 10 randomly sampled initial points in $\mathcal{X}_0$. All trajectories avoid the crescent-shaped unsafe set and converge to the origin (blue star), under the vanishing input $l(t)$ and controller $u(x,t)$.}
    \label{fig:trajectories}
\end{figure}

\bibliographystyle{IEEEtran} 
\bibliography{ref}

\end{document}